\documentclass[12pt]{amsproc}
\newtheorem{theorem}{\sc Theorem}[section]
\newtheorem{lemma}[theorem]{\sc Lemma}

\begin{document}
\title[Centralizers of commutators]{Centralizers of commutators in finite groups}

\thanks{The first and second authors are members of GNSAGA (INDAM), 
and the third author was  supported by  FAPDF and CNPq.}

\author{Eloisa Detomi}
\address{Dipartimento di Ingegneria dell'Informazione - DEI, Universit\`a di Padova, Via G. Gradenigo 6/B, 35121 Padova, Italy} 
\email{eloisa.detomi@unipd.it}
\author{Marta Morigi}
\address{Dipartimento di Matematica, Universit\`a di Bologna\\
Piazza di Porta San Donato 5 \\ 40126 Bologna \\ Italy}
\email{marta.morigi@unibo.it}
\author{Pavel Shumyatsky}
\address{Department of Mathematics, University of Brasilia\\
Brasilia-DF \\ 70910-900 Brazil}
\email{pavel@unb.br}

\date{Received: date / Accepted: date}
\subjclass[2010]{20E45, 20F24, 20F14} 
\keywords{Commutators, centralizers, conjugacy classes}

\maketitle

\begin{abstract} Let $G$ be a finite group. A coprime commutator in $G$ is any element that can be written as a commutator $[x,y]$ for suitable $x,y\in G$ such that $\pi(x)\cap\pi(y)=\emptyset$. Here $\pi(g)$ denotes the set of prime divisors of the order of the element $g\in G$. An anti-coprime commutator is an element that can be written as a commutator $[x,y]$, where $\pi(x)=\pi(y)$. The main results of the paper are as follows.

\noindent If $|x^G|\leq n$ whenever $x$ is a coprime commutator, then $G$ has a nilpotent subgroup of $n$-bounded index. 

\noindent If $|x^G|\leq n$ for every anti-coprime commutator $x\in G$, then $G$ has a subgroup $H$ of nilpotency class at most $4$ such that $[G : H]$ and $|\gamma_4 (H)|$ are both $n$-bounded.

\noindent We also consider finite groups in which the centralizers of coprime, or anti-coprime, commutators are of bounded order.
 \end{abstract}
\section{Introduction}
 
In recent years groups containing many elements lying in conjugacy classes of bounded size have attracted significant interest. In particular, such groups play an important role in the study of commuting probability in finite groups (cf. \cite{dieshu,dmsqjm,as,EloP,eber}). Given a group $G$ and an element $x\in G$, we write $x^G$ for the conjugacy class containing $x$. Of course, the cardinality of $x^G$ is equal to the index $[G:C_G(x)]$. It is a classical result of B. H. Neumann that if $n$ is a positive integer and $|x^G|\leq n$ for all $x\in G$, then the commutator subgroup $G'$ is finite and has $n$-bounded order \cite{bhn}. Throughout this paper we say that a number is $n$-bounded to mean that it is bounded in terms of $n$ only.

An element $x$ of a group $G$ is called a commutator if it can be written as $x=[a,b]=a^{-1}b^{-1}ab$ for suitable $a,b\in G$. The main result of \cite{dieshu} is that if $|x^G|\leq n$ for all commutators in a group $G$, then the second commutator subgroup $G''$ is finite and has $n$-bounded order. The use of probabilistic techniques in \cite{eber} led to further progress in the study of such groups. In particular it was established that if $|x^G|\le n$ for every commutator, then $G$ has a subgroup $H$ of nilpotency class at most $4$ such that the index $[G:H]$ and the order $|\gamma_4(H)|$ are both finite and bounded. 

Let $G$ be a finite group. By a coprime  commutator in $G$ we mean an  element of the form $[x,y]$, where $x$ and $y$ have coprime orders. It is well-known that in any finite group $G$ the coprime commutators generate the nilpotent residual $\gamma_\infty (G)$, that is, the (unique) minimal normal subgroup $N$ such that $G/N$ is nilpotent. In this paper we take interest in finite groups with bounded conjugacy classes containing coprime commutators.

\begin{theorem}\label{main1}
Let $G$ be a finite group in which $|x^G|\leq n$ whenever $x$ is a coprime commutator. Then $G$ has a nilpotent subgroup of $n$-bounded index. 
\end{theorem} 

Remark that under the hypotheses of the above theorem the order of $\gamma_\infty(G)$ can be arbitrarily large. This can be checked by looking at the dihedral groups of order $2t$, where $t$ is odd. Fairly detailed information on the structure of $\gamma_\infty(G)$ will be provided by Theorem \ref{tec} in Section \ref{sec:3}. 

We say that an element $x$ of a finite group is an anti-coprime commutator if $x=[a,b]$ for suitable $a,b\in G$ such that $\pi(a)=\pi(b)$. Here $\pi(g)$ denotes the set of prime divisors of the order of the element $g\in G$.

Note that the subgroup generated by the set of anti-coprime commutators of $G$ is precisely the commutator subgroup of $G$. Indeed, let $N$ be the subgroup generated by all anti-coprime commutators of $G$. Obviously $P'\leq N$ for any Sylow subgroup $P$ of $G$. Further, note that $[x,y,y]=[y^{-x},y]^y$ is an anti-coprime commutator for any $x,y\in G$. Therefore the quotient $G/N$ is nilpotent (see e.g. \cite[12.3.6]{rob}) and has abelian Sylow subgroups. Hence, $G/N$ is abelian and $N=G'$. 

Theorem \ref{main1} happens to be useful in tackling groups with bounded conjugacy classes of anti-coprime commutators. It is somewhat surprising that the conclusion in the next theorem is as strong as in the aforementioned result from \cite{eber}.

\begin{theorem}\label{main-anti} Let $G$ be a finite group in which $|x^G|\leq n$ whenever $x$ is an anti-coprime commutator. Then $G$ has a
subgroup $H$ of nilpotency class at most $4$ such that $[G : H]$ and $|\gamma_4 (H)|$ are both $n$-bounded.
\end{theorem}

We also consider finite groups with small centralizers of coprime and anti-coprime commutators. In an earlier paper \cite{DMS_small} we considered groups in which,  for some group-word $w$ and a positive integer $n$, all nontrivial $w$-values have centralizer of order at most $n$. For some families of group-words it was shown that if $G$ is a finite group in which $w(G)\neq1$ and $|C_G(x)| \le n$ for every nontrivial $w$-value, then the order of $G$ is $n$-bounded. Here we will establish similar results with respect to coprime and anti-coprime commutators.

\begin{theorem}\label{cent_copr}
Let $G$ be a finite group in which the centralizers of nontrivial coprime commutators have order at most $n$. Then either $G$ is nilpotent or it has $n$-bounded order. 
\end{theorem}

\begin{theorem}\label{cent_anti}
Let $G$ be a finite group in which the centralizers of nontrivial anti-coprime commutators have order at most $n$. Then either $G$ is abelian or it has $n$-bounded order. 
\end{theorem}

The next section of this paper collects some preliminary results, Theorems \ref{main1} and \ref{main-anti} are proved in the third and fourth section respectively, and the last section is devoted to the study of finite groups with bounded centralizers of coprime (resp. anti-coprime) commutators.

\section{Preliminaries}

For a group $A$ acting by automorphisms on a group $G$ we use the usual notation for commutators $[g,a]=g^{-1}g^a$ and commutator subgroups $[G,A]=\langle [g,a]\mid g\in G,\;a\in A\rangle$, as well as for centralizers $C_G(A)=\{g\in G\mid g^a=g \text{ for all }a\in A\}$.

We say for short that the group $A$ acts on $G$ coprimely if the orders of $A$ and $G$ are finite and coprime, that is, $(|A|,|G|)=1$. Several well-known facts about coprime actions  will be often used without special references.
\begin{lemma}\cite[Lemmas 4.28--4.29]{isaacs}\label{coprime}
Let a group $A$ act coprimely on a finite group $G$. We have
\begin{itemize}
  \item[\rm (a)] $G=[G,A]C_G(A)$ and  $[G, A, A] = [G, A];$
  \item[\rm (b)] If $N$ is a normal $A$-invariant subgroup of $G$, then $C_{G/N}(A) = C_G(A)N/N$.
\end{itemize}
  \end{lemma}
  
 The following lemma from \cite{bul}  will be helpful. The nilpotency of $G$ cannot be omitted from the hypotheses.
 
\begin{lemma}\cite[Lemma 4.6]{bul} \label{bul} Let $\phi$ be a coprime automorphism of a finite nilpotent group $G$. Then the set of all elements of the form $[g,\phi]$, where $g\in G$, coincides with that of the elements of the form $[g,\phi,\phi]$.
\end{lemma}

Let $Z_\infty(G)$ stand for the last term of the upper central series of a (finite) group $G$. The next result was established in \cite{kos}. It strengthens a classical theorem, due to Baer.

\begin{theorem}\label{ko} Suppose that $Z_\infty(G)$ has finite index $t$ in $G$. Then $\gamma_\infty(G)$ is finite and has $t$-bounded order.
\end{theorem}

The following result is a particular case of Theorem 1.2 in  \cite{qjm2021}. It will play an important role in the proof of Theorem \ref{main1}.

\begin{theorem}\label{qjm} Let $n$ be a positive integer and $G$ a group containing a subgroup $A$ such that $|[g, a]^G|\leq n$ for all $g\in G$ and $a\in A$. Then the commutator subgroup of $[G, A]$ has finite $n$-bounded order.
\end{theorem}

The next observation (cf. \cite[Lemma 1.1.2]{Segal}) will be useful. 
\begin{lemma}\label{finite} Let $X$ be a set of generators of a finite group $G$. Then every element of $G$ can be written as a product of at most $|G|$ elements of $X$. 
\end{lemma}

If $K$ is a subgroup of a finite group $G$, write $$Pr(K, G) = \frac{|\{(x, y) \in K\times G\,\mid\, xy = yx\}|}{|K|\,|G|}.$$
This is the probability that an element of $G$ commutes with an element of $K$ (the relative commutativity degree of $K$ in $G$). The following theorem was established in \cite{EloP}. The particular case where $K=G$ is a classical result due to P. M. Neumann \cite{neumann}.

\begin{theorem}\label{ds} Let $K$ be a subgroup of a finite group $G$ such that $Pr (K,G)=\epsilon>0$. Then there exist a normal subgroup $T$ of $\epsilon$-bounded index in $G$ and a subgroup $B$ of $\epsilon$-bounded index in $K$ such that $[T,B]$ has $\epsilon$-bounded order. 
\end{theorem}

\section{Coprime commutators with bounded conjugacy classes}\label{sec:3}

Remark that if $G$ is as in Theorem \ref{main1}, or Theorem  \ref{main-anti}, the hypotheses are inherited by subgroups and quotients of $G$.   We freely use this property throughout the paper without special references.

Throughout this section $G$ will be  a finite group in which $|x^G|\leq n$ whenever $x$ is a coprime commutator. Recall that the coprime commutators of $G$ generate the nilpotent residual $\gamma_\infty (G)$.

\begin{lemma}\label{key} Suppose $G$ has a normal subgroup $M$ and a subgroup $H$ such that $G=MH$ and $(|M|,|H|)=1$. Then the index $[H:C_H(M)]$ is $n$-bounded.
\end{lemma}
\begin{proof} Since $M=[M,H]C_M(H)$, it follows that $C_H([M,H])=C_H(M)$ and so, without loss of generality, we assume that $M=[M,H]$.  Moreover, we may assume that $C_H(M)=1$ and so we need to show that the order of $H$ is $n$-bounded. Note that the order of a finite group is bounded in terms of the maximum of orders of the abelian subgroups (see for instance \cite[Theorem 5.2]{avisa}) so we can assume that $H$ is abelian. 

By Theorem \ref{qjm} the subgroup $[G,H]'$, which is equal to $M'$, has $n$-bounded order. Thus the subgroup $C_G(M')$ has $n$-bounded index in $G$. We can replace $H$ with $C_H(M')$ and assume that $M'\leq C_M(H)$. 

Now consider the action of $H$ on $M/M'$ and let $C_0 = C_{H}(M/M')$. As $[M,C_0]\leq M'$ and $M'\leq C_M(H)$, we have  that $[M,C_0,C_0]=1$. It follows from Lemma \ref{coprime}  that $[M,C_0]=1$, hence $C_0 \le C_H(M)=1$. Therefore $H$ acts faithfully on $M/M'$ and so, without loss of generality, we can also assume that $M'=1$, that is, $M$ is abelian. 

Thus, both $M$ and $H$ are abelian. Choose arbitrarily two elements $x,y\in G$ and write $x=m_1h_1$ and $y=m_2h_2$ for suitable $m_1,m_2\in M$ and $h_1,h_2\in H$. Taking into account that $M$ and $H$ are abelian we have $$[x,y]=[m_1h_1,m_2h_2]=[m_1,h_2]^{h_1}[h_1,m_2]^{h_2}.$$ Therefore $[x,y]$ is a product of at most two coprime commutators, whence we deduce that $|[x,y]^G|\leq n^2$. Since this holds for arbitrary $x,y\in G$, the main result of \cite{eber} implies that $G$ has a nilpotent subgroup of $n$-bounded index. In particular the Fitting subgroup $F$ of $G$ has $n$-bounded index. Moreover, since $M \le F$ and $|H \cap F|$ is coprime with $|M|$, we deduce that $H \cap F \le C_H(M)=1$. So, $H \cong G/F$ has $n$-bounded order, as claimed. 
\end{proof}

\begin{lemma}\label{simple} Suppose $G$ is a direct product of nonabelian simple groups. Then $|G|\leq n!$.
\end{lemma}
\begin{proof} Let $G=S_1 \times \dots \times S_k$, where each $S_i$ is a nonabelian simple group. By the Feit-Thompson theorem \cite{ft} each factor $S_i$ contains an involution, say $a_i$.  For every $i=1,\dots,k$ choose an odd-order element $b_i\in S_i$ such that $[a_i, b_i] \neq 1$ and set $a=(a_1, \dots , a_k)$ and $b=(b_1, \dots , b_k)$. Note that $[a,b]$ is a coprime commutator. Observe that $C=C_G([a,b])=C_{S_1}([a_1,b_1])\times\dots\times C_{S_k}([a_k,b_k])$. It follows that no simple factor $S_i$ is contained in $C$ and we conclude that $C$ is core-free, that is, the only normal subgroup of $G$ contained in $C$ is the trivial one. On the other hand, since the index of $C$ is at most $n$, the intersection of the conjugates of $C$ has index at most $n!$. Thus, we deduce that $|G|\leq n!$.
\end{proof}

Given a positive integer $e$ and a group $K$, write $K^e$ for the subgroup generated by the $e$th powers of elements of $K$. The exponent of a finite group $K$ is the minimal positive integer $e$ such that $K^e=1$. In their seminal paper \cite{ha-hi} Hall and Higman showed that a finite group of exponent $e$ possesses a normal series of $e$-bounded length all of whose quotients are either nilpotent or isomorphic to a direct product of nonabelian simple groups. This will be used in the next lemma. As usual, the Fitting subgroup of $K$ is denoted by $F(K)$.

\begin{lemma}\label{length} The  exponent of $G/F(G)$ is $n$-bounded. Moreover, the group $G$ possesses a normal series of $n$-bounded length all of whose quotients are either nilpotent or isomorphic to a direct product of nonabelian simple groups.
\end{lemma}
\begin{proof} 
Set $e=n!$.  Note that $G^e$ centralizes each coprime commutator of $G$ and so it centralizes $\gamma_\infty(G)$. Therefore $G^e \le F(G)$. Thus, $G/F(G)$ has exponent dividing $e$.  The claim concerning the normal series in $G$ is immediate from the Hall-Higman theory.
 \end{proof}
Recall that the soluble radical of a finite group $K$ is the (unique) maximal soluble normal subgroup of $K$. If $K$ is soluble, the Fitting height $h(K)$ of $K$ is the length of the shortest normal series all of whose quotients are nilpotent. 
\begin{lemma}\label{soluble}  The soluble radical of $G$ has $n$-bounded index in $G$ and $n$-bounded Fitting height. 
\end{lemma}
\begin{proof} Lemma \ref{length} tells us that the group $G$ possesses a normal series $$1=G_1<G_2<\dots<G_s=G$$ of $n$-bounded length all of whose quotients are either nilpotent or isomorphic to a direct product of nonabelian simple groups. Obviously, the Fitting height of the soluble radical is at most $s$ so we only need to show the index is $n$-bounded. Let $t$ be the number of non-nilpotent quotients in the above series. For each $i$ such that the quotient $G_{i+1}/G_i$ is non-nilpotent in the usual way define the centralizer in $G$ via $$C_i=\{g\in G\ | \ [G_{i+1},g]\leq G_i\}.$$ Lemma \ref{simple} tells us that any non-nilpotent quotient has order at most $n!$. Therefore each $C_i$ has index at most $n!!$. It is easy to see that the intersection $\cap C_i$ is precisely the soluble radical of $G$ and it has index at most $(n!!)^t$.
\end{proof}

Now we are ready to prove that if $G$ is a finite group in which $|x^G|\leq n$ whenever $x$ is a coprime commutator, then the index of the Fitting subgroup of $G$ is $n$-bounded. Throughout, for a set of primes $\pi$, we write $O_{\pi}(K)$ to denote the maximal normal $\pi$-subgroup of a group $K$.

\begin{proof}[Proof of Theorem \ref{main1}] In view of Lemma \ref{soluble} we can assume that $G$ is soluble and its Fitting height $h(G)$ is $n$-bounded. So we will use induction on $h(G)$.  By induction, $G/F(G)$ has a nilpotent subgroup of $n$-bounded index and so we can assume that $G$ is metanilpotent, that is, $h(G)=2$. We know from Lemma \ref{length}  that $G/F(G)$ has exponent dividing $n!$. In particular there are less than $n$ primes, say $p_1,\dots,p_s$ dividing the order of $G/F(G)$. For $i=1,\dots,s$ let $P_i$ be a $p_i$-Sylow subgroup of $G$. Note that the subgroups $F(G)P_i$ are normal in $G$ and their product equals $G$. 

Therefore it is sufficient to consider the case where $G=F(G)P$ for a Sylow $p$-subgroup $P\leq G$. Set $M=O_{p'}(F(G))$ and note that $G=MP$. Now an application of Lemma \ref{key} completes the proof.
\end{proof} 

As was mentioned in the introduction, under the hypotheses of Theorem \ref{main1} the order of $\gamma_\infty(G)$ can be arbitrarily large. The next theorem provides detailed information on the structure of $\gamma_\infty(G)$. 

\begin{theorem}\label{tec} Let $G$ be a finite group in which $|x^G|\leq n$ whenever $x$ is a coprime commutator. Then 
 \begin{enumerate} 
 \item\label{1} every element of $\gamma_\infty(G)$ is a product of $n$-boundedly many  coprime commutators; 
 \item\label{2} there exists a positive real number $\epsilon$, depending on $n$, such that $Pr (\gamma_\infty(G), G)\geq\epsilon$; 
 \item \label{3} there exist a normal subgroup $T$ of $n$-bounded index in $G$ and a subgroup $B$ of $n$-bounded index in $\gamma_\infty(G)$ such that $[T,B]$ has $n$-bounded order. 
 \end{enumerate}
 \end{theorem}
\begin{proof} 
Note that statement (\ref{3}) is a  direct consequence of statement   (\ref{2}) and Theorem \ref{ds}.  Moreover, statement   (\ref{2}) is straightforward from  (\ref{1}). Indeed, if there exists an $n$-bounded integer $t$ such that every  element $x$ of $\gamma_\infty(G)$ is a product of at most $t$  coprime commutators, then  $x$   has at most $(n^t)$-conjugates. It follows that 
\[
 Pr (\gamma_\infty(G), G)= \sum_{x \in \gamma_\infty(G)} \frac{|C_G(x)|}{|\gamma_\infty(G)| \, |G|} \ge 1/n^t, 
\]
  as claimed in (\ref{2}). 
  
 So it suffices to show that every element of $\gamma_\infty(G)$ is a product of $n$-boundedly many  coprime commutators. 
 
Let $F=F(G)$ be the Fitting subgroup of $G$. By  Theorem \ref{main1}, the quotient $G/F$ has $n$-bounded order. Let $\pi$ be the set of primes dividing the order of $G/F$. Of course the size of $\pi$ here is $n$-bounded. Choose Sylow subgroups $P_1,\dots,P_u$, one for each prime in $\pi$.
We can pick $n$-boundedly many elements $a_1, \dots ,a_s\in\cup_{1\leq j\leq u}P_j$ such that $G=\langle a_1,\dots,a_s,F\rangle$. Thus, the elements $a_i$ are of prime-power order and those whose orders are divisible by the same prime generate a $p$-subgroup. For each $i= 1, \dots, s$ let $p_i$ the prime dividing the order of $a_i$ (so $p_i$ is not necessarily different from $p_j$ when $i\neq j$) and let $R_i=O_{p_i'}(F)$. 

Now we fix an element $a=a_i$, and, for short, write $p=p_i$ and $R=R_i$. We have $R=[R, a]C_R( a)$ and $[R,a,a] =[R,a] $.  By Theorem \ref{qjm} the commutator subgroup $[R,a]'$ has $n$-bounded order. 

Note that $R$ is normal in $G$ and $F=R \times O_p(F)$. The subgroup $[R,a]$ is normal in $R$ and it is centralized by $O_p(F)$ since $[R,a] \le O_{p'}(F)$. It follows that $F$ normalizes $[R,a]$. As $G/F$ has $n$-bounded order, we deduce that the normalizer of $[R,a]$ has $n$-bounded index in $G$. Obviously, this implies that also the normalizer of the commutator subgroup $[R,a]'$ has $n$-bounded index in $G$. 

Thus  the product, say $N_a$, of the conjugates of $[R,a]'$ is a product of $n$-boundedly many subgroups, all normal in $F$. We conclude that $N_a$ has $n$-bounded order. Moreover, note that $[R,a]$ is generated by the coprime commutators $[r,a]$ with $r\in R$, and so $[R,a]'$ is generated by products of two coprime commutators 
  $[[r,a], [s,a]]$, with $r, s \in R$. So  $N_a$ can be generated by a set of elements that are products of two coprime commutators. As $N_a$ has $n$-bounded order, it follows from Lemma \ref{finite} that every element of  $N_a$  is a product of boundedly many  coprime commutators. 

The image of $[R,a]$ in $G/N_a$ is abelian. Since $[R,a]=[R,a,a]$, using standard commutator identities deduce that every element of $[R,a]$ is the product of a coprime commutator and an element of $N_a$. In particular, every element of $[R,a]$ is the product of $n$-boundedly many  coprime commutators. Let $M_a$ denote the normal closure of $[R,a]$ in $G$. Note that $M_a$ is a product of $n$-boundedly many conjugates of $[R,a]$, which are normal subgroups of $F$. We therefore deduce that every element of  $M_a$ is a product of boundedly many  coprime commutators. 

Arguing in this manner for all $a_i$ and taking into consideration that $i\leq s$, where $s$ is  $n$-bounded, we conclude that every element of
 the  subgroup 
 \[B= \prod_{i=1, \dots, s} [R_i, a_i]^G \] 
   is a product of boundedly many  coprime commutators. 
  
Consider now the quotient group $\bar G=G/B$. In this quotient ${[\bar R_i, \bar a_i]}=1$ for every $i=1,\dots,s$. If $\bar P$ is a Sylow $p$-subgroup of $\bar F$, then $\bar G/C_{\bar G}(\bar P)$ is a $p$-group. It follows that $\bar P\leq Z_\infty(\bar G)$. This happens for all Sylow subgroups of $\bar F$ so we deduce that $\bar F \le Z_{\infty}(\bar G)$ and therefore $Z_{\infty}(\bar G)$ has $n$-bounded index in $\bar G$. 
 It follows from Theorem \ref{ko}  that $\gamma_\infty(\bar G)$ has $n$-bounded oder. 
 
 Therefore $\gamma_\infty(G)/B$ has $n$-bounded order. Now we apply Lemma \ref{finite} to deduce that  every element of $\gamma_\infty(G)$ is a product of $n$-boundedly many  coprime commutators and an element of $B$. Since we have already shown that every element of $B$ is  a product of $n$-boundedly many  coprime commutators, the same holds for the elements of $\gamma_\infty(G)$. The proof is now complete. 
 \end{proof}

\section{Anti-coprime commutators with bounded conjugacy classes}

Throughout this section,  $G$ will be  a finite group in which $|x^G|\leq n$ whenever $x$ is an anti-coprime commutator. We wish to show that $G$ has a subgroup $H$ of nilpotency class at most $4$ such that $[G : H]$ and $|\gamma_4 (H)|$ are both $n$-bounded. We state the next lemmas without proofs because they are easy modifications of the corresponding facts from the previous section.

\begin{lemma}\label{simple-anti} Suppose $G$ is a direct product of nonabelian simple groups. Then $|G|\leq n!$.
\end{lemma}

\begin{lemma}\label{length-anti}  The  exponent of $G/F(G)$ divides  $n!$. Moreover, the soluble radical of $G$ has $n$-bounded index in $G$ and $n$-bounded Fitting height.
\end{lemma}

Theorem  \ref{main-anti} will now be proved as a consequence of Theorem \ref{main1}. 

\begin{proof}[Proof of Theorem  \ref{main-anti}] 
 By Lemma \ref{length-anti} we can assume that $G$ is soluble with $n$-bounded Fitting height $h(G)$.
	
 First use induction on $h(G)$ to show that $G$ has a nilpotent subgroup of $n$-bounded index. The result is obvious if $h(G)=1$. Otherwise, assume by induction that $G/F(G)$ has a nilpotent subgroup of $n$-bounded index and so reduce to the case when $G$ is metanilpotent. Set $F=F(G)$. We know from Lemma \ref{length-anti} that $G/F$ has exponent dividing $n!$. In particular there are less than $n$ primes dividing the order of $G/F$. So it is sufficient to consider the case where $G=FP=O_{p'}(F)P$ for a Sylow $p$-subgroup $P\leq G$.
 
 Observe that if $[x,y]\in G$ is a coprime commutator, then one of the elements $x,y$ belongs to $O_{p'}(F)$ while the other belongs to a conjugate of $P$. Assume that $x\in O_{p'}(F)$ and $y\in P$. Lemma \ref{bul} tells us that there is an element $u\in O_{p'}(F)$ such that \[ [x,y] =[u,y,y]=[y^{-u},y]^y,\] which is an anti-coprime commutator.
 Thus, any coprime commutator in $G$ is also an anti-coprime commutator. Therefore $|g^G|\leq n$ whenever $g$ is a coprime commutator. So by Theorem \ref{main1} the group $G$ has a nilpotent normal subgroup of $n$-bounded index.

Therefore we can additionally assume that $G$ is nilpotent. Remark that in a nilpotent group every commutator is an anti-coprime commutator.  Hence, $|x^G|\le n$ for every commutator $x\in G$.  The result from \cite{eber} mentioned in the introduction now tells us that $G$ has a subgroup $H$ of nilpotency class at most $4$ such that the index $[G:H]$ and the order $|\gamma_4(H)|$ are both finite and bounded.
\end{proof}

\section{Commutators with centralizers of bounded order}

In this section we will prove Theorems \ref{cent_copr} and \ref{cent_anti}. Thus we consider finite groups in which the centralizers of nontrivial coprime commutators (anti-coprime commutators, respectively) have order at most $n$. We will need a useful lemma which implies that our hypotheses are inherited by quotients.

\begin{lemma}\label{total} \cite[Lemma 2.12]{KK}
Let $G$ be a finite group and $x\in G$. Then $|C_{G/N}(xN)| \le |C_{G}(x)|$ for every normal subgroup $N$ of $G$.
 \end{lemma}

The following theorem is due to Hartley. 
 
\begin{theorem}\cite[Theorem A]{H}\label{H-thmA}
If $G$ is a finite group containing an element $x$ with $|C_G (x)| \le n$, then the index of the soluble radical in $G$ is $n$-bounded.\end{theorem}

We will also need some tools from \cite{DMS_small}, suitably adapted to our hypotheses. For the reader's convenience we include the proofs.

\begin{lemma}\label{abel} 
Let $x$ be an element a group $G$ such that  $|C_G(x)|\le n$. If $x$ is contained in an abelian normal subgroup  $A$ of $G$, then $|G|$ is $n$-bounded.  
\end{lemma}
\begin{proof}
 We have $A \le C_G(A) \le C_G(x)$. Therefore $A$ and $C_G(A)$ have at most $n$ elements. Since $G/C_G(A)$ admits a natural embedding in the group of automorphisms of $A$, the order of  $G/C_G(A)$ is bounded by $n!$. Thus $|G| \le n \cdot {n!}$.  
\end{proof}

We will first deal with groups in which the centralizers of nontrivial coprime commutators have order at most $n$. 

\begin{lemma}\label{reduction:copr}
 Let $G$ be a finite group in which the centralizers of nontrivial coprime commutators have order at most $n$. Then $G$ has a nilpotent normal subgroup of $n$-bounded index which contains no nontrivial coprime commutators.
\end{lemma}
\begin{proof}
 By Theorem \ref{H-thmA},  $G$ has a soluble normal subgroup $H$ of $n$-bounded index. 
 Let 
 \[ 1=A_0 \le A_1 \le \dots \le A_t=H\] 
  be a characteristic series of $H$ with the property that each section $A_{i+1}/A_i$ is abelian (e.g. the derived series of $H$). Let $i$ be the greatest integer such that $A_i$ contains no nontrivial coprime commutators. Note that $A_i$ is nilpotent.

If $i=t$, then the result holds since $A_t=H$ has $n$-bounded index in $G$. Otherwise, $i < t$ and, by the maximality of $i$, the subgroup $A_{i+1}$ contains a 
nontrivial coprime commutator $x$. By Lemma  \ref{total}, the centralizer of the element $xA_i$ in the quotient group $G/A_i$ has order at most $n$. It follows from Lemma \ref{abel}  that $|G/A_i|$ is $n$-bounded, as desired.
\end{proof}

\begin{proof}[Proof of Theorem  \ref{cent_copr}] By Lemma \ref{reduction:copr},  $G$ has a nilpotent normal subgroup $N$ of $n$-bounded index which contains no nontrivial coprime commutators. Observe that $N\le Z_\infty(G)$. Indeed, let $P$ be a Sylow $p$-subgroup of $N$. As $P$ is a normal subgroup of $G$ containing no nontrivial coprime commutators, $C_G(P)$ contains every $p'$-element of $G$ and so $G/C_G(P)$ is a $p$-group. Therefore $P\leq Z_\infty(G)$. Since this happens for every Sylow subgroup of $N$, it follows that  $N\le Z_\infty(G)$.

As $N$ has $n$-bounded index in $G$, the same holds for $Z_\infty(G)$. Therefore, by  Theorem \ref{ko},  we get that $\gamma_\infty(G)$ has $n$-bounded oder. 
If $\gamma_\infty(G)=1$ then $G$ is nilpotent. Otherwise, let $x\in G$  be a nontrivial coprime commutator and let $C=C_G(\gamma_\infty(G))$. As $C \le C_G(x)$, the order of $C$ is at most $n$. Since $G/C$ embeds into the   group of automorphisms of $\gamma_\infty(G)$,  the order of $G/C$ is bounded in terms of $|\gamma_\infty(G)|$, which in turn is $n$-bounded. Thus $|G|$ is  $n$-bounded.
\end{proof}

Now we turn to finite groups  in which the centralizers of nontrivial anti-coprime commutators have order at most $n$. 

\begin{proof}[Proof of Theorem  \ref{cent_anti}] Let $G$ be a finite group in which the centralizers of nontrivial anti-coprime commutators have order at most $n$. We want to prove that either $G$ is abelian or the order of $G$ is $n$-bounded.  Note that all values in $G$ of the $2$-Engel word are anti-coprime commutators, as $[x,y,y]=[y^{-x},y]^y$. It follows from  \cite{DMS_small} that either $G$ has $n$-bounded order or it is a $2$-Engel group. In the latter case $G$ is nilpotent and so every commutator  in $G$ is an anti-coprime commutator. Now we deduce from \cite{DMS_small} that either $G$ is abelian or the order of $G$ is $n$-bounded. This concludes the proof.
\end{proof}

\end{document}